\pdfoutput=1
\documentclass[11pt,letterpaper]{article}

\usepackage[margin=1in]{geometry}
\usepackage[T1]{fontenc}
\usepackage{lmodern}
\usepackage{microtype}
\usepackage{amsmath,amssymb,amsthm}
\usepackage[round,authoryear]{natbib}
\usepackage[colorlinks=true,citecolor=blue,linkcolor=blue,urlcolor=blue]{hyperref}
\hypersetup{pdftitle={Counterexamples to a conjectured Schatten norm inequality},pdfauthor={Zhekai Pang}}

\newtheorem{conjecture}{Conjecture}
\newtheorem{theorem}{Theorem}

\newcommand{\M}{\mathrm M}
\newcommand{\R}{\mathbb R}
\newcommand{\C}{\mathbb C}
\newcommand{\Tr}{\operatorname{Tr}}

\title{Counterexamples to a conjectured Schatten norm inequality}
\author{Zhekai Pang\thanks{Universitat Pompeu Fabra, Barcelona, Spain. Email: zhekai.pang@upf.edu.}}
\date{}

\begin{document}

\maketitle

\begin{abstract}
For \(p\geq 1\) and \(m\geq 2\), let \(c_p(m)\) be the least constant such that
\[
 \left\|\sum_{k=1}^m A_k\right\|_p
 \leq c_p(m)\left\|\sum_{k=1}^m |A_k|\right\|_p
\]
for all square complex matrices. Tang and Zhang recently conjectured a closed formula for \(c_p(m)\), obtained from a common-left equiangular rank-one family. We disprove that formula for every \(m\geq2\) and every \(1<p<2\).
\end{abstract}



\section{Introduction}

Let \(\M_n(\C)\) denote the set of \(n\) by \(n\) complex matrices. For \(X\in\M_n(\C)\), write
\[
 |X|=(X^*X)^{1/2},
 \qquad
 \|X\|_p=(\Tr |X|^p)^{1/p}
 \quad (1\leq p<\infty).
\]
The comparison between \(|A+B|\) and \(|A|+|B|\) is a fundamental problem in matrix analysis. Thompson proved that there are unitaries \(U,V\) such that
\[
 |A+B|\leq U|A|U^*+V|B|V^*
\]
\citep{Thompson1976}. For positive semidefinite matrices, Bourin and Uchiyama established a subadditivity inequality for nonnegative concave functions and all unitarily invariant norms \citep{BourinUchiyama2007}. Bourin later obtained the corresponding modulus inequality for normal matrices \citep{Bourin2010}.

For arbitrary matrices, Lee proved the dimension-free estimate
\[
 \|A_1+\cdots+A_m\|\leq \sqrt m\,\||A_1|+\cdots+|A_m|\|
\]
for every unitarily invariant norm and asked for the best constant for Schatten norms \citep{Lee2012}. Audenaert and Kittaneh recorded the more general concave-function version as Problem~1 in their survey \citep{AudenaertKittaneh2017}. In the two-summand Frobenius case, Lee conjectured the sharp constant \(\sqrt{(1+\sqrt2)/2}\). This was proved by Lin and Zhang \citep{LinZhang2022}; a later proof was given by Zhang \citep{Zhang2025}.

Tang and Zhang extended the linear problem to \(m\) summands and proved the sharp Frobenius inequality
\[
 \left\|\sum_{k=1}^m A_k\right\|_2
 \leq \sqrt{\frac{1+\sqrt m}{2}}
 \left\|\sum_{k=1}^m |A_k|\right\|_2.
\]
They also proved the dimension-free bound
\[
 c_p(m)\leq (\sqrt m)^{1-1/p}
\]
and proposed a closed formula for the optimal constant \citep{TangZhang2026}. Their conjectural value is generated by an equiangular rank-one family in which all matrices have the same left singular vector.

This paper shows that the proposed formula is false throughout the range \(1<p<2\). We use one equiangular system on the left and another on the right. The parameters are chosen so that the additional singular values in the numerator raise the resulting Schatten \(p\)-quotient above the conjectured value.

\section{The conjecture}

The following is Conjecture~3.1 of \citet{TangZhang2026}.

\begin{conjecture}[Tang--Zhang]
Let \(p>1\) and \(m\geq2\). Let \(c_p(m)\) be the smallest number such that
\[
 \left\|\sum_{k=1}^m A_k\right\|_p
 \leq c_p(m)\left\|\sum_{k=1}^m |A_k|\right\|_p
 \qquad\text{for all }A_1,\ldots,A_m\in\M_n(\C).
\]
Then
\begin{equation}\label{eq:conjectured-constant}
 c_p(m)=
 \frac{\sqrt{x_{p,m}(x_{p,m}+m-1)}}
 {(x_{p,m}^p+m-1)^{1/p}},
 \qquad x_{p,m}>0\text{ solves }x^p-2x-(m-1)=0.
\end{equation}
\end{conjecture}

\section{Counterexamples}

\begin{theorem}\label{thm:main}
Conjecture~1 is false for every integer \(m\geq2\) and every \(1<p<2\). More precisely, let \(x=x_{p,m}\) be the number in Conjecture~1 and set
\[
 y=x^{p/(2-p)}.
\]
Let \(e_1,\ldots,e_m\) be the standard basis of \(\R^m\), and let
\(\mathbf 1=(1,\ldots,1)^T\). For \(k=1,\ldots,m\), define
\begin{align}
 u_k&=\sqrt{\frac{m}{y+m-1}}
 \left(e_k+\frac{\sqrt y-1}{m}\mathbf 1\right),\label{eq:uk}\\
 v_k&=\sqrt{\frac{m}{x+m-1}}
 \left(e_k+\frac{\sqrt x-1}{m}\mathbf 1\right),\label{eq:vk}\\
 A_k&=u_kv_k^T.\label{eq:Ak}
\end{align}
Then
\begin{equation}\label{eq:strict-counterexample}
 \left\|\sum_{k=1}^m A_k\right\|_p
 >
 \frac{\sqrt{x(x+m-1)}}{(x^p+m-1)^{1/p}}
 \left\|\sum_{k=1}^m |A_k|\right\|_p.
\end{equation}
\end{theorem}

\begin{proof}
We first verify that the vectors in \eqref{eq:uk} and \eqref{eq:vk} are unit vectors. For each \(k\),
\[
 \left\|e_k+\frac{\sqrt y-1}{m}\mathbf 1\right\|_2^2
 =1+\frac{2(\sqrt y-1)}{m}+\frac{(\sqrt y-1)^2}{m}
 =\frac{y+m-1}{m}.
\]
It follows from \eqref{eq:uk} that \(\|u_k\|_2=1\). Replacing \(y\) by \(x\) gives \(\|v_k\|_2=1\). For \(j\neq k\), direct multiplication also gives
\[
 u_j^Tu_k=\frac{y-1}{y+m-1},
 \qquad
 v_j^Tv_k=\frac{x-1}{x+m-1},
\]
so both vector families are equiangular.

Since \(A_k=u_kv_k^T\) and both vectors are unit,
\[
 A_k^TA_k=v_k u_k^Tu_kv_k^T=v_kv_k^T.
\]
The matrix \(v_kv_k^T\) is an orthogonal projection, and therefore
\begin{equation}\label{eq:absolute-rank-one}
 |A_k|=(A_k^TA_k)^{1/2}=v_kv_k^T.
\end{equation}

Let \(J=\mathbf 1\mathbf 1^T\). We have
\[
 \sum_{k=1}^m e_ke_k^T=I,
 \qquad
 \sum_{k=1}^m e_k\mathbf 1^T=J,
 \qquad
 \sum_{k=1}^m \mathbf 1e_k^T=J.
\]
Expanding \eqref{eq:absolute-rank-one} and using these identities gives
\begin{align*}
 \sum_{k=1}^m |A_k|
 &=\frac{m}{x+m-1}
 \left[I+
 \left(\frac{2(\sqrt x-1)}{m}
 +\frac{(\sqrt x-1)^2}{m}\right)J\right]\\
 &=\frac{m}{x+m-1}\left(I+\frac{x-1}{m}J\right).
\end{align*}
Since \(J\mathbf 1=m\mathbf 1\) and \(Jw=0\) for every \(w\perp\mathbf 1\), the preceding matrix acts as multiplication by \(mx/(x+m-1)\) on \(\operatorname{span}\{\mathbf 1\}\). On \(\mathbf 1^\perp\), it acts as multiplication by \(m/(x+m-1)\), and this eigenvalue has multiplicity \(m-1\). The matrix is positive semidefinite, so its singular values are precisely these eigenvalues. Consequently,
\begin{equation}\label{eq:denominator}
 \left\|\sum_{k=1}^m |A_k|\right\|_p^p
 =\left(\frac{mx}{x+m-1}\right)^p
 +(m-1)\left(\frac{m}{x+m-1}\right)^p
 =\left(\frac{m}{x+m-1}\right)^p(x^p+m-1).
\end{equation}

We next compute the other norm. Expanding \eqref{eq:Ak} and summing over \(k\), the coefficient of \(J\) is
\[
 \frac{\sqrt x-1}{m}+\frac{\sqrt y-1}{m}
 +\frac{(\sqrt x-1)(\sqrt y-1)}{m}
 =\frac{\sqrt{xy}-1}{m}.
\]
Thus
\[
 \sum_{k=1}^m A_k
 =\frac{m}{\sqrt{(x+m-1)(y+m-1)}}
 \left(I+\frac{\sqrt{xy}-1}{m}J\right).
\]
This matrix is positive semidefinite. It acts as multiplication by
\(m\sqrt{xy}/\sqrt{(x+m-1)(y+m-1)}\) on
\(\operatorname{span}\{\mathbf 1\}\). On \(\mathbf 1^\perp\), it acts as multiplication by
\(m/\sqrt{(x+m-1)(y+m-1)}\), and this eigenvalue has multiplicity \(m-1\). Hence
\begin{align}
 \left\|\sum_{k=1}^m A_k\right\|_p^p
 &=\left(\frac{m\sqrt{xy}}{\sqrt{(x+m-1)(y+m-1)}}\right)^p
 +(m-1)\left(\frac{m}{\sqrt{(x+m-1)(y+m-1)}}\right)^p\notag\\
 &=\left(\frac{m}{\sqrt{(x+m-1)(y+m-1)}}\right)^p
 \bigl((xy)^{p/2}+m-1\bigr).\label{eq:numerator}
\end{align}

Dividing \eqref{eq:numerator} by \eqref{eq:denominator} gives
\begin{equation}\label{eq:exact-quotient}
 \frac{\left\|\sum_{k=1}^m A_k\right\|_p^p}
 {\left\|\sum_{k=1}^m |A_k|\right\|_p^p}
 =
 \left(\frac{x+m-1}{y+m-1}\right)^{p/2}
 \frac{(xy)^{p/2}+m-1}{x^p+m-1}.
\end{equation}

We now compare \eqref{eq:exact-quotient} with the \(p\)-th power of the conjectured constant. Since \(1<p<2\) and \(y=x^{p/(2-p)}\),
\begin{align*}
 x^{p/2}\bigl[(y+m-1)^{p/2}-y^{p/2}\bigr]
 &=\frac p2 x^{p/2}\int_y^{y+m-1}t^{p/2-1}\,dt\\
 &<\frac p2(m-1)x^{p/2}y^{p/2-1}\\
 &=\frac p2(m-1)\\
 &<m-1.
\end{align*}
Therefore
\[
 (xy)^{p/2}+m-1>x^{p/2}(y+m-1)^{p/2}.
\]
Using this inequality in \eqref{eq:exact-quotient},
\begin{align*}
 \frac{\left\|\sum_{k=1}^m A_k\right\|_p^p}
 {\left\|\sum_{k=1}^m |A_k|\right\|_p^p}
 &=\left(\frac{x+m-1}{y+m-1}\right)^{p/2}
 \frac{(xy)^{p/2}+m-1}{x^p+m-1}\\
 &>\left(\frac{x+m-1}{y+m-1}\right)^{p/2}
 \frac{x^{p/2}(y+m-1)^{p/2}}{x^p+m-1}\\
 &=\frac{\bigl(x(x+m-1)\bigr)^{p/2}}{x^p+m-1}.
\end{align*}
Taking \(p\)-th roots proves \eqref{eq:strict-counterexample}, and hence
\[
 c_p(m)>\frac{\sqrt{x(x+m-1)}}{(x^p+m-1)^{1/p}}.
\]
\end{proof}

\end{document}